\documentclass{amsart}
\usepackage[english]{babel}
\usepackage[cp1250]{inputenc}
\usepackage{amssymb,amsthm,amsfonts}
\usepackage{verbatim}
\usepackage{url}

\newtheorem{theorem}{Theorem}[section]
\newtheorem{lemma}[theorem]{Lemma}

\theoremstyle{definition}
\newtheorem{definition}[theorem]{Definition}

\newtheorem{remark}[theorem]{Remark}

\newcommand{\PP}{\mathbb P}

\newcommand{\cE}{\mathcal{E}}
\newcommand{\cF}{\mathcal{F}}
\newcommand{\cG}{\mathcal{G}}
\newcommand{\cH}{\mathcal{H}}
\newcommand{\cI}{\mathcal{I}}

\newcommand{\cK}{\mathcal{K}}
\newcommand{\cL}{\mathcal{L}}

\newcommand{\cO}{\mathcal{O}}
\newcommand{\cQ}{\mathcal{Q}}
\newcommand{\cR}{\mathcal{R}}

\newcommand{\cP}{\mathcal{P}}
\newcommand{\cT}{\mathcal{T}}

\begin{document}

\title{Rank 2 ACM bundles on Complete intersection Calabi-Yau threefolds}
\author{Matej Filip}
\address{University of Ljubljana, Faculty of Math.~and Phys., Dept.~of Math.}
\email{matej.filip@student.fmf.uni-lj.si}

\date{\today}

\begin{abstract}

The aim of this paper is to classify indecomposable rank 2 arithmetically Cohen-Macaulay (ACM) bundles on compete intersection Calabi-Yau (CICY) threefolds and prove the existence of some of them. New geometric properties of the curves corresponding to rank 2 ACM bundles (by Serre correspondence) are obtained. These follow from minimal free resolutions of curves in suitably chosen fourfolds (containing Calabi-Yau threefolds as hypersurfaces). Also the existence of an indecomposable vector bundle of higher rank on a CICY threefold of type (2,4) is proved.

\end{abstract}



\maketitle

\section{Introduction}
Curves and vector bundles on a general threefold $X\subset \PP^n$ have been considered as an important tool for the description of the geometry of $X$.

The existence of ACM bundles is linked to the existence of some 
arithmetically Cohen-Macaulay curves, via Serre corespondence between rank 2 bundles on threefolds and locally complete intersection subcanonical curves (see e.g. \cite{har3} or \cite{mad2}).

We will consider smooth Calabi-Yau threefolds, i.e. smooth 3-dimensional projective varieties with trivial canonical class. In particular, if $X=X_{d_1\cdots d_k}\subset \PP^{3+k}$ is a complete intersection of hypersurfaces of degrees $d_1, ..., d_k$ then $X$ is called a CICY threefold.
By adjunction formula (see e.g. \cite[p. 59]{ful}), we obtain the following CICY threefolds:
\begin{itemize}
\item quintic threefold $X_5$ in $\PP^4$,
\item complete intersection $X_8$ of type (2,4) in $\PP^5,$ 
\item complete intersection $X_9$ of type (3,3) in $\PP^5,$
\item complete intrersection $X_{12}$ of type (2,2,3) in $\PP^6$,
\item complete intersection $X_{16}$ of type (2,2,2,2) in $\PP^7$.  
\end{itemize}
Chiantini and Madonna classify ACM rank 2 bundles on a general quintic threefolds \cite[p. 247]{chi} and prove the existence of some of them. They use the tools of deformation theory from Kley \cite{kle}. Rao and Kumar prove the existence of all ACM bundles from  \cite[p. 247]{chi}, where some of the proofs involve computer calculations.
They proved the existence of bundles using the Pfaffian matrix representation of the quintic threefold, so their result cannot be generalized to other CICY threefolds. A list of indecomposable ACM bundles on CICY threefolds is given in Madonna's paper \cite{mad1}, however we believe it is incorrect (see Theorem \ref{theorem1}).

Recently Knutsen \cite{knu} and Yu \cite{yu} explored the existence of smooth isolated curves on CICY threefolds. Their results and the theory of elliptic and canonical curves (see e.g. Babbage \cite{bab}, Eisenbud \cite{eis}, Noether \cite{noe}, Fisher \cite{fis}) will help us to prove the existence of indecomposable rank 2 bundles on general CICY threefolds.

For the existence of rank 4 indecomposable vector bundles on a quintic threefold see Madonna \cite{mad3}.
  
The main result of this paper is
\begin{theorem}\label{theorem1}
Let $X_r\subset \PP^{3+k}$ (where $k=\left\lfloor \frac{r}{4} \right\rfloor$) be a CICY threefold and let $\cE$ be an indecomposable ACM rank 2 vector bundle on it. Then the normalization of $\cE$ has one of the following Chern classes:
\begin{itemize}
\item $c_1=-2$, $c_2=1$,
\item $c_1=-1$, $c_2=2$,
\item $c_1=0$, $3\leq c_2 \leq 4+k$,
\item $c_1=1$, $4\leq c_2 \leq 6+2k$ and $c_2$ is even,
\item $c_1=2$, $c_2\leq 7+2k+r$
\item $c_1=3$, $c_2=8+2k+2r$
\item $c_1=4$,
\begin{itemize}
\item $c_2=30$ if $r=5$,
\item $c_2=44$ if $r=8$,
\item $c_2=48$ if $r=9$,
\item $c_2=62$ if $r=12$,
\item $c_2=80$ if $r=16$.
\end{itemize}
\end{itemize}

We prove the existence of $\cE$ on a general $X_r$ for $c_1=-2,-1$ and $0$, for all possible $r$ and $c_2$ listed above, except  $c_1=0$, $c_2=3$ on $X_{16}$. There also exists bundles for $r=8$ with $c_1=1$, $c_2=6$ and $c_1=1$, $c_2=10$; for $r=9$ with $c_1=1$, $c_2=6$ and $c_1=1$, $c_2=10$; for $r=12$ with $c_1=1$, $c_2=8$ and $c_1=1, c_2=12$.
\end{theorem}

Section 2 includes definitions, notations, the Grothendieck-Riemann-Roch formula for rank 2 bundle on CICY threefolds and states the Serre corespondence.
Our Theorem \ref{izjave} in Section 3 explicitely relates minimal resolutions of arithmetically Gorenstein curves and the corresponding rank 2 bundles. This is a generalization of Faenzi, Chiantini result \cite{chifae} for rank 2 bundles on surfaces.
In Section 4 we classify bundles using case by case analysis. We obtain some intersesting properties of the corresponding curves and some minimal free resolutions which lead to the proof of Theorem \ref{theorem1} in Section 5. We also prove the existence of an indecomposable bundle of higher rank on $X_8$.  

\section{Generalities}

We work over the field of complex numbers. 
A vector bundle on a projective scheme $X$ is a locally free coherent sheaf on $X$. 
We denote by $\cO_X$ the structure sheaf of $X$ and for any vector bundle $\cE$ we write $\cE(n)=\cE\otimes \cO_X(n)$.

Since Picard group of $X_r$ is isomorphic to $\mathbb{Z}$ and second Chern class $c_2(\cE)$ is a multiple of the class of a line, we identify Chern classes $c_1(\cE)$, $c_2(\cE)$ and line bundles with integers.
We write $c_i$ for $c_i(\cE)$.
If $\cE$ is a rank 2 vector bundle on $X_r$, we have 
$$c_1(\cE(n))=c_1(\cE)+2n,$$ 
$$c_2(\cE(n))=c_2(\cE)+rnc_1(\cE)+rn^2.$$

\begin{lemma}
The Grothendieck-Riemann-Roch formula (GRR) for a rank 2 bundle $\cE$ on a CICY threefold $X_r$ is
\begin{equation}\label{gro}
\chi (\cE)=\frac{r}{6}c_1^3-\frac{c_1c_2}{2}+\frac{c_1}{12}(12(k+4)-2r), \mathrm{where}~~ k=\left\lfloor \frac{r}{4} \right\rfloor.
\end{equation} 
\end{lemma}
\begin{proof}
Let $h$ denote the generator of the Picard group and let $l$ be the class of a line.  
From the Grothendieck-Riemann-Roch formula (see e.g. page 431 in \cite{har}) we get
$$\deg(\mathrm{ch}(\cE)\mathrm{td}(\cT_X))_3=\frac{1}{6}(c_1^3-3c_1c_2)+\frac{1}{4}d_1(c_1^2-2c_2)+\frac{1}{12}(d_1^2+d_2)c_1+\frac{1}{12}d_1d_2,$$ 
where $c_i=c_i(\cE)$,  $d_i=c_i(\cT_X)$, and $\cT_X$ is the tangent sheaf of $X$.
By the adjunction formula (\cite[pg. 59]{ful}) we have $d_1=0$ and in case $r=5$ we obtain $d_2=10h^2$, in the case $r=8$ we obtain $d_2=7h^2$, in the case $r=9$ we obtain $d_2=6h^2$, in the case $r=12$ we obtain $d_2=5h^2$ and in the case $r=16$ we obtain $d_2=4h^2$. We have $h^2=r\cdot l$ and after identifying the Chern classes with integers we get \eqref{gro}. 
\end{proof}

\begin{definition}
Let $\cI_V$ be the saturated ideal of the closed subscheme $V$ of $\PP^n$. Then $V$ is \emph{arithmetically Cohen-Macaulay} (ACM) if 
$$\dim \mathbb{C}[x_0,...,x_n]/\cI_V=\mathrm{depth}~\mathbb{C}[x_0,...,x_n]/\cI_V.$$
\end{definition}
It holds (see e.g. \cite[Lemma 1.2.3]{mig}) that if $\dim V=r\geq 1,$ then $V$ is ACM if and only if $(M^i)(V)=0$, for $1\leq i\leq r.$ Here $(M^i)(V)$ is the deficiency module of $V$, defined as the i-th cohomology module of the ideal sheaf of $V$: 
$$(M^i)(V)=H^i_*(\cI_V).$$

A locally complete intersection projective variety $V\subset \PP^n$  is \emph{subcanonical} if the canonical sheaf $\omega_V$ is isomorphic to $\cO_V(k)$, for some integer $k$.

If $V\subset \PP^n$ is ACM and the last bundle in the  minimal free resolution of $\cI_V$ is a line bundle, then we call $V$ \emph{arithmetically Gorenstein} (AG). 
Note that an ACM variety $V$ is AG if and only if it is subcanonical. If there exists a variety $X\supseteq V$ such that the last bundle in the  minimal free resolution of $\cI_V$ in $X$ is a line bundle then we say that $V$ is \emph{AG in} $X$.     


A sheaf $\cE$ on a $k$-dimensional projective variety $X$ is \emph{arithmetically Cohen-Macaulay} (ACM) if $h^i(\cE(n))=0$ for $i=1,...,k-1$ and for all $n\in \mathbb{Z}$.

We say that a bundle $\cE$ on $X$ is \emph{normalized} if the number 
$$b(\cE):=\max \{n~|~H^0(X,\cE(-n))\ne 0\}$$ 
is equal to zero. Clearly, the normalization of $\cE$ is $\cE(-b(\cE))$. 

The Serre correspondence between bundles and curves (see e.g. \cite{har3} or \cite{mad2}) is the following: 

\begin{theorem}.
Let $X$ be a smooth 3-dimensional projective variety. A curve $C\subset X$ occurs as the zero-locus of a section of a rank 2 vector bundle $\cE$ on $X$, if $C$ is a local complete intersection and subcanonical.
More precisely, for any fixed invertible sheaf $\cL$ on $X$ with $h^1(\cL^{\vee})=h^2(\cL^{\vee})=0$, there exists a bijection between the following set of data:
\begin{enumerate}
\item the set of triples $(\cE, s, \phi)$, where $s\in H^0(X,\cE)$ and $\phi: \wedge^2\cE \rightarrow \cL$ is an isomorphism,  modulo equivalence relation defined in \cite{har3}.
\item the set of pairs $(C,\cE)$, where $C$ is a locally complete intersection curve in $X$ and $\cL\otimes \omega_X \otimes \cO_C$ and $\omega_C$ are isomorphic.
\end{enumerate} 

A normalized bundle $\cE$ has a section whose zero locus $C$ is a curve and we have an exact sequence 
\begin{equation}\label{exs}
0\longrightarrow \cO_X\longrightarrow \cE \longrightarrow\cI_C(c_1(\cE)),
\end{equation}
where $\cI_C$ is an ideal sheaf of $C$ on $X$.
The curve $C$ is ACM if and only if $\cE$ is ACM and $C$ has degree $c_2(\cE)$. Moreover, if $X$ is a Calabi-Yau threefold, the genus of $C$ is $\frac{c_1(\cE)c_2(\cE)}{2}+1$.
\end{theorem}

\section{Minimal resolution of ACM curves on complete intersection threefolds}

Let $\cE$ be a vector bundle of rank 2 on a complete intersection threefold $X\subset \PP^{3+k}$ of type $(d_1,...,d_k)$. Denote by $Y_i$ the complete intersection 
fourfold of type $(d_1,..,d_{i-1},d_{i+1},...,d_k)$ which contains $X$ and let $a_j$ be the degrees of minimal generators for $\cE$ in $Y_i$. By abuse of notation write $\cE$ for the sheaf $i_{*}\cE$, where $i:X\hookrightarrow Y_i$. 
Clearly, for $k=1$ take $Y_1=\PP^4$.
From the Auslander-Buchsbaum formula we get a resolution 
\begin{equation}\label{eq:1}
0\longrightarrow \cF \longrightarrow \cG \longrightarrow \cE \longrightarrow 0,
\end{equation}   
where $\cG=\bigoplus_{j=1}^k\cO_{Y_i}(-a_j)$ and $\cF$ is an ACM bundle on $Y_i$. 

From now on we will assume that $\cF$ splits.
If $Y_i$ is $\PP^4$ (which is the case when $X$ is a quintic threefold), then by Horrock's criterion $\cF$ splits.  

After dualising the exact sequence \eqref{eq:1}, we get  
\begin{equation}\label{eq:dual}
0\longrightarrow \cH om(\cG,\cO_{Y_i}) \longrightarrow \cH om(\cF,\cO_{Y_i}) \longrightarrow \cE xt^1_{\cO_{Y_i}}(\cE,\cO_{Y_i}) \longrightarrow 0.
\end{equation}
Indeed,  $\cH om_{\cO_{Y_i}}(\cE,\cO_{Y_i})=0$ since the support of $\cE$ is in $X$ and $\cE xt^i(\cF,\cO_{Y_i})=\cE xt^i(\cG,\cO_{Y_i})=0$, for $i>0$ by \cite[Proposition III.6.3]{har} since $\cF$ splits.

The same argument as in \cite{rao2} for a quintic yields that even when $Y_i$ is not equal to $\PP^4$, we have an isomorphism $\cE xt^1_{Y_i}(\cE, Y_i)\cong \cE^{\vee}(d_i)$. Indeed, applying the functor $\cH om(\cE, \cdot)$ on the exact sequence
$$0 \to \cO_{Y_i}\to \cO_{Y_i}(d_i)\to \cO_X(d_i)\to 0,$$ 
we obtain 
\begin{align*}
0&\longrightarrow &\cH om_{\cO_{Y_i}}(\cE,\cO_{Y_i})& \longrightarrow &\cH om_{\cO_{Y_i}}(\cE,\cO_{Y_i}(d_i)) &\longrightarrow &\cH om_{\cO_{Y_i}}(\cE,\cO_X(d_i))& \longrightarrow \\
&\longrightarrow &\cE xt^1_{\cO_{Y_i}}(\cE,\cO_{Y_i})& \stackrel{f}{\longrightarrow} &\cE xt^1_{\cO_{Y_i}}(\cE,\cO_{Y_i}(d_i)) &\longrightarrow &\cE xt^1_{\cO_{Y_i}}(\cE,\cO_X(d_i))& \longrightarrow \cdots
\end{align*} 
where $f$ is multiplication by the defining polynomial of $X$ in $Y_i$.
As above we see that $\cH om_{\cO_{Y_i}}(\cE,\cO_{Y_i})=\cH om_{\cO_{Y_i}}(\cE,\cO_{Y_i}(d_i))=0$. Since $\cE xt^1_{\cO_{Y_i}}(\cE,\cO_{Y_i})$ and $\cE xt^1_{\cO_{Y_i}}(\cE,\cO_{Y_i}(d_i))$ are both supported on $X$, we obtain an isomorphism 
$$\cE xt^1_{\cO_{Y_i}}(\cE, \cO_{Y_i})\cong \cE^{\vee}(d_i)\cong \cE(d_i-c_1).$$
Thus we see that resolutions \eqref{eq:1} and \eqref{eq:dual} twisted by $c_1-d_i$ are equivalent resolutions for $\cE$.
In the case of a quintic threefold Beauville \cite{bea} showed the existence of a resolution of the following type
\begin{equation}\label{quintica}
0 \longrightarrow \cG^{\vee}(c_1-5)\longrightarrow \cG \longrightarrow \cE \longrightarrow 0.
\end{equation}
A similar resolution for CICY threefolds will appear in the following theorem.

\begin{theorem}\label{izjave}
Assume that a complete intersection threefold $X$ contains a subcanonical local complete intersection curve $C$, which is AG in at least one of the $Y_i$. Then the ideal sheaf of $C$ has a minimal resolution

\begin{equation}\label{eq:3}
0\longrightarrow P_2 \longrightarrow P_1 \longrightarrow P_0 \longrightarrow \cI_{C,Y_i} \longrightarrow 0,
\end{equation}
where
$$P_0=\bigoplus_{j=1}^{2b+1}\cO_{Y_i}(-r_j), ~~P_1=\bigoplus_{j=1}^{2b+1}\cO_{Y_i}(r_j-c),~~P_2=\cO_{Y_i}(-c).$$
Here $c=(\sum_{j=1}^{2b+1}r_j)/b$ and $r_j$ are the degrees of minimal generators of $\cI_{C,Y_i}$ and $2b+1$ is the number of these generators.
Assume also that $c_1(\cE)=c-d_i$, where $\cE$ is a normalized rank 2 ACM bundle on $X$, corresponding to the curve $C$.
Then $\cE$ has a minimal resolution  

\begin{equation}\label{eq:res4}
0\longrightarrow L_1 \longrightarrow L_0 \longrightarrow \cE \longrightarrow 0,
\end{equation}
where 
$$L_1=\cO_{Y_i}(c-2d_i)\oplus \left (\bigoplus_{j=1}^{2b+1}\cO_{Y_i}(r_j-d_i)\right ),$$
$$L_0=\cO_{Y_i}\oplus  \left (\bigoplus_{j=1}^{2b+1}\cO_{Y_i}(-r_j+c-d_i)\right ).$$
\end{theorem}

From now on, we fix the fourfold $Y_i$ satisfying the assumptions of Theorem \ref{izjave} and denote it by $Y$. We also denote the degree of the defining polynomial of $X$ in $Y$ by $d$ instead of $d_i$. Observe that \eqref{eq:res4} is of the following type 
\begin{equation}\label{eq:pomembno}
0 \to \cL_0^{\vee}(c-2d)\to \cL_0 \to \cE \to 0.
\end{equation}

\begin{proof}[Proof of Theorem \ref{izjave}]
By Eisenbud and Buchsbaum \cite{buceis} every ideal sheaf of an AG curve in $Y$ has a resolution of type \eqref{eq:3}.
From \eqref{eq:3} we obtain two short exact sequences:
\begin{equation}\label{1}
0\longrightarrow \cO_Y(-c) \longrightarrow \cP_1 \longrightarrow \cK \longrightarrow 0,
\end{equation}
\begin{equation}\label{2}
0\longrightarrow \cK \longrightarrow \cP_0 \longrightarrow \cI_{C,Y} \longrightarrow 0.
\end{equation}
We also have the following exact sequences:
\begin{equation}\label{3}
0\longrightarrow \cO_Y(-d) \longrightarrow \cO_Y \longrightarrow \cO_X \longrightarrow 0,
\end{equation}
\begin{equation}\label{4}
0\longrightarrow \cO_Y(-d) \longrightarrow \cI_{C,Y} \longrightarrow \cI_{C,X} \longrightarrow 0,
\end{equation}
\begin{equation}\label{5}
0\longrightarrow \cO_X(d-c) \longrightarrow \cE(d-c) \longrightarrow \cI_{C,X} \longrightarrow 0.
\end{equation}

Let $\cQ$ be the kernel of the surjective map $\cP_0\rightarrow \cI_{C,X}$. Thus we have
\begin{equation}\label{6}
0\longrightarrow \cQ \longrightarrow \cP_0 \longrightarrow \cI_{C,X} \longrightarrow 0.
\end{equation}  
By the snake lemma applied to \eqref{2} and \eqref{6}, $\cQ$ fits into 
\begin{equation}\label{7}
0\longrightarrow \cK \longrightarrow \cQ \longrightarrow \cO(-d) \longrightarrow 0.
\end{equation} 
Next, apply $\mathrm{Hom}(\cP_0,\cdot)$ to \eqref{5}.
Since $\mathrm{Ext}^1(\cP_0,\cO_X(d-c))=0$, the map $\cP_0\to \cI_{C,X}$ lifts to the map $\cP_0 \to \cE(d-c)$ and thus we can connect \eqref{5} and \eqref{6}.
The mapping cone determines the surjectivity of $\cP_0\oplus \cO_X(d-c)\rightarrow \cE(d-c)$. Thus $\cP_0\oplus \cO_Y(d-c)\rightarrow \cE(d-c)$ is also surjective and we have
\begin{equation}\label{8}
0\to \cR \to \cP_0\oplus \cO_Y(d-c) \to \cE(d-c) \to 0,
\end{equation}
where $\cR$ denotes the kernel.

From the snake lemma applied to \eqref{8} and \eqref{3} twisted by $d-c$, we obtain
\begin{equation}\label{9}
0 \to \cO_Y(-c) \to \cR \to \cQ \to 0.
\end{equation}
Again apply the snake lemma to \eqref{1} and \eqref{9} to get 
$$0 \to \cP_1 \to \cR \to \cO_Y(-d) \to 0,$$
which splits since $\mathrm{Ext}^1(\cO(-d), \cP_1)=0$. 
Thus we proved
\begin{equation}\label{10}
0 \to \cP_1\oplus \cO_Y(-d)\to \cP_0\oplus \cO_Y(d-c) \to \cE(d-c) \to 0.
\end{equation}
Observe, that twisting \eqref{10} by $c-d$ gives \eqref{eq:res4}.
It can be easily verified that the obtained resolution is minimal.
\end{proof}

\begin{remark}\label{indija}
In \cite[Lemma 2, Lemma 5, Proposition 2]{bis} some minimal resolution of ACM rank 2 bundles on cubic and quartic threefolds are determined. If we apply Theorem \ref{izjave} on these cases we obtain the same resolutions except in the case of \cite[Lemma 5]{bis}, where $X$ is a quartic threefold and the curve $C$ corresponding to the bundle $\cE$ with $c_1=2$, $c_2=8$ is of type (2,2,2). We have a minimal resolution 
$$0\longrightarrow \cO_{\PP^4}(-6) \longrightarrow \cO_{\PP^4}(-4)^3 \longrightarrow \cO_{\PP^4}(-2)^3  \longrightarrow \cI_C \longrightarrow 0$$
and by Theorem \ref{izjave} ($c=6$, $d_i=4$, $c_1=c-d_i$, $r_j=2$, for $j=1,..., 3$) we get 
$$0\longrightarrow \cO_{\PP^4}(-2)^4 \longrightarrow \cO_{\PP^4}^4 \longrightarrow \cE \longrightarrow 0.$$

However, in \cite{bis} the obtained resolution of $\cE$ is of the form 
$$0\longrightarrow \cO_{\PP^4}(-2)^4\oplus \cO_{\PP^4}(-1)^k \longrightarrow \cO_{\PP^4}^4\oplus \cO_{\PP^4}(-1)^k \longrightarrow \cE \longrightarrow 0,$$
where $k\in \{0,2,4\}$. It is then deduced that $k=0$ implies $\cE$ is 0-regular, so $k\ne 0$. We believe this is incorrect, since $\cE$ in this case is 1-regular. Theorem \ref{izjave} immediately implies that the case $k=0$ is the correct choice.
\end{remark}

\section{Classification and existence of ACM curves on CICY threefolds}

By \cite[Theorem 3.9]{mad2} a normalized rank 2 ACM bundle on a CICY threefold splits unless $-5<-c_1(\cE)<3$.
Recall from the Serre correspondence that if $\cE$ is an indecomposable ACM rank 2 bundle, then the corresponding curve is ACM of degree $c_2(\cE)$ and genus $c_1(\cE)c_2(\cE)/2+1$.  
In this section we will classify the indecomposable rank 2 ACM bundles on CICY threefolds.

\subsection{Quintic threefold}
In this subsection we write $\cO$ for $\cO_{\PP^4}$. 
For a quintic threefold $X$ in $\PP^4$ the minimal resolution of a curve can be determined from the minimal resolution of the corresponding bundle:
\begin{theorem}\label{izjave2}
Let $\cE$ be a normalized indecomposable rank 2 vector bundle on $X$ with a minimal resolution 
\begin{equation}\label{nor}
0\longrightarrow \cO(c_1(\cE)-5)\oplus \left (\bigoplus_{i=1}^{2b+1} \cO(r_i+c_1(\cE)-5)\right ) \longrightarrow \cO\oplus \left (\bigoplus_{i=1}^{2b+1}\cO(-r_i)\right ) \longrightarrow \cE \longrightarrow 0.
\end{equation}
Then the corresponding curve has a minimal resolution 
$$0\longrightarrow \cO(-c_1(\cE)+5) \longrightarrow \bigoplus_{i=1}^{2b+1}\cO(r_i-5) \longrightarrow  \bigoplus_{i=1}^{2b+1}\cO(-r_i-c_1(\cE))  \longrightarrow \cI_C \longrightarrow 0.$$
\end{theorem}
\proof
The proof is similar to the proof of \cite[Theorem 2.1]{chifae} so we omit it (note that $c_1(\cE)=c-5$, where $c$ is the integer from Theorem \ref{theorem1}).
\endproof
The classification of indecomposable ACM rank 2 bundles on a quintic can be found in \cite{chi}. In particular, the lower bound for $c_2$ is also given: $11\leq c_2 \leq 14$.
Theorem \ref{izjave2} allows us to determine geometric properties of the curves that correspond to these bundles. Many of these properties are already established in \cite{chi} and \cite{rao}. Here we only desribe the new ones.

Let $\cE$ be a bundle with $c_1=2$, $c_2=11$. From GRR we have $h^0(\cE)=4$, $h^0(\cE(1))=18$, $h^0(\cE(2))=52$ and because a minimal resolution of $\cE$ is of type \eqref{quintica}, we have a minimal resolution 
$$0\longrightarrow \mathcal{O}(-1)^2\oplus \cO(-3)^4 \longrightarrow \cO(-2)^2 \oplus \cO^4 \longrightarrow \cE \longrightarrow 0.$$
From Theorem \ref{izjave2} the minimal resolution of the coresponding curve $C$ is 
$$0\longrightarrow \cO(-7) \longrightarrow \mathcal{O}(-5)^3\oplus \cO(-3)^2 \longrightarrow \cO(-2)^3 \oplus \cO(-4)^2 \longrightarrow \cI_C \longrightarrow 0$$ 
where the degree matrix (see \cite{her}) of $C$ is 

$$
\left [
\begin{array}{ccccc}
3&3&3&1&1 \\
3&3&3&1&1 \\
3&3&3&1&1 \\
1&1&1&0&0 \\
1&1&1&0&0
\end{array}
\right ]
.
$$
By \cite[Theorem 1.2]{her} $C$ is singular.

In the case of $c_1$ even, the minimal resolution of $\cE$ can be uniquely determined as above. 
When $c_1$ is odd, the rank of the direct summands in a minimal resolution of $\cE$ cannot be determined, similar to the problem in Remark \ref{indija}.
For example, without taking into account the geometry of the corresponding curve, an ACM bundle $\cE$ with 
$c_1=-1$ and $c_2=2$ has a minimal resolution
$$0\longrightarrow \cO(-6)\oplus \cO^2(-4)\oplus \cO(-3)^j \longrightarrow \cO \oplus \cO(-2)^2\oplus \cO(-3)^j \longrightarrow \cE \longrightarrow 0.$$
Since the corresponding curve is a conic, Theorem \ref{izjave} implies that $j=1$.

In the sequel we collect minimal resolutions of all possible indecomposable normalized ACM bundles $\cE$ of rank 2. 
By the above methods all these resolutions are uniquely determined, except for $c_1=3$. 
We have:

\begin{itemize}
\item  $ c_1=-2$, $ c_2=1$
$$0\longrightarrow \cO(-3)\longrightarrow \cO(-2)^3 \longrightarrow \cO(-1)^3 \longrightarrow \cI_C \longrightarrow 0,$$
$$0\longrightarrow \cO(-7)\oplus \cO(-4)^3 \longrightarrow \cO \oplus \cO(-3)^3 \longrightarrow \cE \longrightarrow 0,$$

\item $c_1=-1$, $c_2=2$
$$0\longrightarrow \cO(-4)\longrightarrow \cO(-2)\oplus \cO(-3)^2 \longrightarrow \cO(-2) \oplus \cO(-1)^2 \longrightarrow \cI_C \longrightarrow 0,$$
$$0\longrightarrow \cO(-6)\oplus \cO^2(-4)\oplus \cO(-3)^1 \longrightarrow \cO \oplus \cO(-2)^2\oplus \cO(-3)^1 \longrightarrow \cE \longrightarrow 0,$$

\item $c_1=0$, $c_2=3$
$$0\longrightarrow \cO(-5)\longrightarrow \cO(-4)^2\oplus \cO(-2) \longrightarrow \cO(-1)^2 \oplus \cO(-3) \longrightarrow \cI_C \longrightarrow 0,$$
$$0\longrightarrow \mathcal{O}(-5)\oplus \cO(-4)^2 \oplus \cO(-2) \longrightarrow \cO \oplus \cO(-1)^2\oplus \cO(-3) \longrightarrow \cE \longrightarrow 0.$$

\item $c_1=0$, $c_2=4$
$$0\longrightarrow \cO(-5)\longrightarrow \cO(-4)\oplus \cO(-3)^2 \longrightarrow \cO(-1) \oplus \cO(-2)^2 \longrightarrow \cI_C \longrightarrow 0,$$ 
$$0\longrightarrow \mathcal{O}(-3)^2\oplus \cO(-5) \oplus \cO(-4) \longrightarrow \cO \oplus \cO(-1)\oplus \cO(-2)^2 \longrightarrow \cE \longrightarrow 0.$$

\item $c_1=0$, $c_2=5$
\begin{equation}\label{g1d5}
0\longrightarrow \cO(-5)\longrightarrow \cO(-3)^5 \longrightarrow \cO(-2)^5 \longrightarrow \cI_C \longrightarrow 0,
\end{equation}
$$0\longrightarrow \mathcal{O}(-5)\oplus \cO(-3)^5 \longrightarrow \cO \oplus \cO(-2)^5 \longrightarrow \cE \longrightarrow 0.$$

\item $c_1=1$, $c_2=4$
$$0\longrightarrow \cO(-6)\longrightarrow \cO(-5)^2\oplus \cO(-2) \longrightarrow \cO(-1)^2 \oplus \cO(-4) \longrightarrow \cI_C \longrightarrow 0,$$
$$0\longrightarrow \mathcal{O}(-4)^3\oplus \cO(-1) \longrightarrow \cO^3 \oplus \cO(-3) \longrightarrow \cE \longrightarrow 0,$$ 

\item $c_1=1$, $c_2=6$
$$0\longrightarrow \cO(-6)\longrightarrow \cO(-5)\oplus \cO(-4)\oplus \cO(-3) \longrightarrow \cO(-1) \oplus \cO(-2)\oplus \cO(-3) \longrightarrow \cI_C \longrightarrow 0,$$
$$0\longrightarrow \mathcal{O}(-4)^2\oplus \cO(-3)\oplus \cO(-2) \longrightarrow \cO^2 \oplus \cO(-1)\oplus \cO(-2) \longrightarrow \cE \longrightarrow 0,$$

\item $c_1=1$, $c_2=8$
$$0\longrightarrow \cO(-6)\longrightarrow \cO(-4)^3 \longrightarrow \cO(-2)^3 \longrightarrow \cI_C \longrightarrow 0,$$
$$0\longrightarrow \mathcal{O}(-3)^3\oplus \cO(-4) \longrightarrow \cO \oplus \cO(-1)^3 \longrightarrow \cE \longrightarrow 0$$

\item $c_1=2$, $c_2=11$
$$0\longrightarrow \cO(-7) \longrightarrow \mathcal{O}(-5)^3\oplus \cO(-3)^2 \longrightarrow \cO(-2)^3 \oplus \cO(-4)^2 \longrightarrow \cI_C \longrightarrow 0,$$ 
$$0\longrightarrow \mathcal{O}(-1)^2\oplus \cO(-3)^4 \longrightarrow \cO(-2)^2 \oplus \cO^4 \longrightarrow \cE \longrightarrow 0,$$

\item $c_1=2$, $c_2=12$
$$0\longrightarrow \cO(-7)\longrightarrow \cO(-5)^2\oplus \cO(-4) \longrightarrow \cO(-2)^2 \oplus \cO(-3) \longrightarrow \cI_C\longrightarrow 0,$$
$$0\longrightarrow \mathcal{O}(-3)^3\oplus \cO(-2) \longrightarrow \cO^3 \oplus \cO(-1) \longrightarrow \cE \longrightarrow 0,$$

\item $c_1=2$, $c_2=13$
$$0\longrightarrow \cO(-7)\longrightarrow \cO(-5)\oplus \cO(-4)^4 \longrightarrow \cO(-2) \oplus \cO(-3)^4 \longrightarrow \cI_C \longrightarrow 0,$$
$$0\longrightarrow \mathcal{O}(-3)^2\oplus \cO(-2)^4 \longrightarrow \cO^2 \oplus \cO(-1)^4 \longrightarrow \cE \longrightarrow 0,$$

\item $c_1=2$, $c_2=14$
$$0\longrightarrow \cO(-7)\longrightarrow \cO(-4)^7 \longrightarrow \cO(-3)^7 \longrightarrow \cI_C \longrightarrow 0,$$
$$0\longrightarrow \mathcal{O}(-3)\oplus \cO(-2)^7 \longrightarrow \cO^2 \oplus \cO(-1)^7 \longrightarrow \cE \longrightarrow 0.$$

\item $c_1=3$, $c_2=20$
$$0\longrightarrow \cO(-8)\longrightarrow \cO(-5)^4\oplus \cO(-4)^k \longrightarrow \cO(-3)^4 \oplus \cO(-4)^k \longrightarrow \cI_C \longrightarrow 0,$$
$$0\longrightarrow \mathcal{O}(-2)^5\oplus \cO(-1)^k \longrightarrow \cO^5 \oplus \cO(-1)^k \longrightarrow \cE \longrightarrow 0,$$
where $k$ is odd.

\item $c_1=4$, $c_2=30$
$$0\longrightarrow \cO(-9)\longrightarrow \cO(-5)^9 \longrightarrow \cO(-4)^9 \longrightarrow \cI_C \longrightarrow 0,$$
$$0\longrightarrow \mathcal{O}(-1)^{10} \longrightarrow \cO^{10} \longrightarrow \cE \longrightarrow 0.$$

\end{itemize}

\subsection{CICY of type (2,4)}

Let $X_8$ be a general CICY threefold of type (2,4).
The only attempt of classification of indecomposable rank 2 bundles known to the autor can be found in \cite{mad1}.

\begin{itemize}

\item $c_1=-2$

Bundle $\cE$ has a section whose 0-locus $C$ is a curve. From the exact sequence 
$$0\longrightarrow \cO_{X_8}\longrightarrow \cE \longrightarrow \cI_C(-2) \longrightarrow 0,$$
we obtain $h^3(\cE)=h^0(\cE(2))=20$, so $\chi(\cE)=-19$. From GRR we have $\chi(\cE)=-20+c_2$, so $c_2=1$ and the corresponding curve is a line.

Since a line exists on $X_8$ (see e.g. \cite{knu}), then by Serre correspondence a bundle with $c_1=-2$ and $c_2=1$ exists on $X_8$.
We easily see that the minimal resolution of $\cE$ is not of type \eqref{eq:pomembno} on either of the two fourfolds containing $X_8$.

\item $c_1=-1$

As above we have an exact sequence
$$0\longrightarrow \cO_{X_8}\longrightarrow \cE \longrightarrow \cI_C(-1) \longrightarrow 0$$
and thus $h^3(\cE)=h^0(\cE(1))=6$. From GRR we have $-5=\chi(\cE)=-6+\frac{c_2}{2}$, so $c_2=2$.
From the above sequence we compute $h^0(\cI(1))=23-20=3$, so $C$ is a plane conic.

This $C$ has a minimal resolution  
$$0\longrightarrow \cO_Y(-3)\longrightarrow \cO_Y(-2)^3 \longrightarrow \cO_Y(-1)^3 \longrightarrow \cI_C \longrightarrow 0,$$
where $Y$ is a fourfold of degree 2 in $\PP^5$. Theorem \ref{izjave} gives a minimal resolution of $\cE$: 
\begin{equation}\label{X8}
0\longrightarrow \mathcal{O}_Y(-5)\oplus \cO_Y(-3)^3 \longrightarrow \cO_Y \oplus \cO_Y(-2)^3 \longrightarrow \cE \longrightarrow 0.
\end{equation}

\item $c_1=0$

We start with an exact sequence
$$0\longrightarrow \cO_{X_8}\longrightarrow \cE \longrightarrow \cI_C \longrightarrow 0.$$
From GRR we obtain $h^0(\cE(1))=12-c_2$, thus we have $c_2=6-h^0(\cI_C(1))$, which gives four possibilities for $c_2$: 3, 4, 5, 6. 
The only case when $C$ is AG in one of the two fourfolds containing $X_8$ is $c_2=4$. In this case $C$ is a space curve of type (2,2).

A minimal resolution of a curve $C$ is 
$$0\longrightarrow \cO_Y(-4) \longrightarrow \cO_Y(-3)^2 \oplus \cO_Y(-2) \longrightarrow \cO_Y(-1)^2 \oplus \cO_Y(-2) \longrightarrow \cI_C \longrightarrow 0,$$
where $Y$ is a fourfold of degree 2 in $\PP^5$.
From Theorem \ref{izjave} we get 
$$0\longrightarrow \cO_Y(-4)\oplus \cO_Y(-3)^2\oplus \cO_Y(-2) \longrightarrow \cO_Y \oplus \cO_Y(-1)^2\oplus \cO_Y(-2) \longrightarrow \cE \longrightarrow 0.$$

\item $c_1=1$

We have
$$0\longrightarrow \cO_{X_8} \longrightarrow \cE \longrightarrow \cI_C(1) \longrightarrow 0.$$
Because $h^3(\cE)=h^0(\cE(-1))=0$, then $1+h^0(\cI_C(1))=h^0(\cE)=\chi(\cE)$.
From GRR follows $c_2=10-2h^0(\cI_C(1))$. So we have four choices for $c_2$, which are 4,6,8,10. 

If $c_2=4$ the corresponding curve is a plane quartic with a resolution 
$$0\longrightarrow \cO_Y(-3)\longrightarrow \cO_Y(-2)^3 \longrightarrow \cO_Y(-1)^3 \longrightarrow \cI_C \longrightarrow 0,$$ 
where $Y$ is a fourfold of degree 4 in $\PP^5$. Theorem \ref{izjave} yields
$$0\longrightarrow \cO_Y(-1)^4 \longrightarrow \cO_Y^4 \longrightarrow \cE \longrightarrow 0.$$

If $c_2=6$ the corresponding curve is a complete intersection of type (2,3) and has a resolution 
$$0\longrightarrow \cO_Y(-5)\longrightarrow \cO_Y(-4)^2\oplus \cO_Y(-2) \longrightarrow \cO_Y(-1)^2 \oplus \cO_Y(-3) \longrightarrow \cI_C \longrightarrow 0,$$ 
where $Y$ is a fourfold of degree 2 in $\PP^5$ and by Theorem \ref{izjave}
$$0\longrightarrow \cO_Y(-3)^3\oplus \cO_Y(-1) \longrightarrow \cO_Y^3 \oplus \cO_Y(-2) \longrightarrow \cE \longrightarrow 0.$$

If $c_2=8$ the corresponding curve is a complete intersection of type (2,2,2) and has a resolution 
$$0\longrightarrow \cO_Y(-5)\longrightarrow \cO_Y(-4)\oplus \cO_Y(-3)^2 \longrightarrow \cO_Y(-2)^2 \oplus \cO_Y(-1) \longrightarrow \cI_C \longrightarrow 0$$
and as before
$$0\longrightarrow \cO_Y(-3)^2\oplus \cO_Y(-2)^2 \longrightarrow \cO_Y^2 \oplus \cO_Y(-1)^2 \longrightarrow \cE \longrightarrow 0,$$ 
where $Y$ is a fourfold of degree 2 in $\PP^5$.

If $c_2=10$ the corresponding curve is canonical of genus 6. We will show the existence of this bundle in the next section, which will also give us an indecomposable bundle of higher rank.

\item $c_1=2$

As before we have 
$$0\longrightarrow \cO_{X_8} \longrightarrow \cE \longrightarrow \cI_C(2) \longrightarrow 0.$$
Beacause $h^3(\cE)=h^0(\cE(-2))=0$, then $1+h^0(\cI_C(2))=h^0(\cE)=\chi(\cE)$.
With GRR we see $c_2=19-h^0(\cI_C(2))$. So we have $c_2\leq 19$.

If $c_2=16$, then the corresponding ACM curve is a complete intersection of type (2,2,2,2).
A minimal resolution of $\cI_C$ is 
$$0\longrightarrow \cO_Y(-6)\longrightarrow \cO_Y(-4)^3 \longrightarrow \cO_Y(-2)^3 \longrightarrow \cI_C \longrightarrow 0,$$
where $Y$ is a fourfold of degree 2 in $\PP^5$
and Theorem \ref{izjave} gives
$$0\longrightarrow \cO_Y(-2)^4 \longrightarrow \cO_Y^4 \longrightarrow \cE \longrightarrow 0.$$

\item $c_1=3$

As above we have 
$$0\longrightarrow \cO_{X_8} \longrightarrow \cE \longrightarrow \cI_C(3) \longrightarrow 0$$
and thus $h^3(\cE(-1))=h^0(\cE(-2))=0$. GRR implies $c_2=28$.

\item $c_1=4$
As above we have $h^3(\cE(-1))=0$ and therefore $c_2=44$.
\end{itemize}

Using similar methods as in the case of CICY of type (2,4), we will obtain the remaining cases CICY threefolds of types (3,3), (2,2,3) and (2,2,2,2).

\subsection{CICY of type (3,3)}
Let $X_9$ be a general CICY of type (3,3) and $Y$ a fourfold of degree 3 in $\PP^5$.

\begin{itemize}

\item $c_1=-2$

We have $c_2=1$ and the corresponding curve is a line. The existence of a line on $X_9$ was showed in \cite{knu}.

\item $c_1=-1$

We have $c_2=2$ and the corresponding curve is a conic. The existence of a conic on $X_9$ can also be found in \cite{knu}.

\item $c_1=0$

We have $6+h^0(\cI_C(1))=12-c_2$. This gives four possible choices for $c_2$, which are 3,4,5,6. 
The only case when $C$ is AG in $Y$ is $c_2=3$. In this case the corresponding curve is a plane cubic. We have
$$0\longrightarrow \cO_Y(-3)\longrightarrow \cO_Y(-2)^3 \longrightarrow \cO_Y(-1)^3 \longrightarrow \cI_C \longrightarrow 0$$   
and
\begin{equation}\label{X9}
0\longrightarrow \mathcal{O}_Y(-3)\oplus \cO_Y(-2)^3 \longrightarrow \cO_Y \oplus \cO_Y(-1)^3 \longrightarrow \cE \longrightarrow 0.
\end{equation}

\item $c_1=1$

We have $6-\frac{c_2}{2}=1+h^0(I_C(1))$ thus by GRR $c_2=10-2h^0(I(1))$. We get four options for $c_2$: 4,6,8,10.
The only case when $C$ is AG in $Y$ is $c_2=6$.
In this case the corresponding curve is a complete intersection of type (2,3).
Suitable resolutions are
$$0\longrightarrow \cO_Y(-4)\longrightarrow \cO_Y(-3)^2\oplus\cO_Y(-2) \longrightarrow \cO_Y(-1)^2\oplus \cO_Y(-2) \longrightarrow \cI_C \longrightarrow 0,$$ 
$$0\longrightarrow \cO_Y(-2)^3\oplus \cO_Y(-1) \longrightarrow \cO_Y^3\oplus \cO_Y(-1) \longrightarrow \cE \longrightarrow 0.$$

\item $c_1=2$

Because $h^3(\cE)=h^0(\cE(-2))=0$, we have $1+h^0(\cI_C(2))=h^0(\cE)=\chi(\cE)$.
From GRR we get $c_2=20-h^0(\cI_C(2))$ and therefore $c_2\leq 20$.

\item $c_1=3$
We have $h^3(\cE(-1))=h^0(\cE(-2))=0$ and by GRR we obtain $c_2=30$.

\item $c_1=4$
We have $h^3(\cE(-1))=0$ and by GRR we obtain $c_2=48$.

\end{itemize}

\subsection{CICY of type (2,2,3) and (2,2,2,2)}

Using the above methods concludes the classification of indecomposable rank 2 bundles listed in Theorem \ref{theorem1}. 
For $c_1=-2$ and $c_1=-1$, Knutsen \cite{knu} proved the existence of a line and a conic on $X_{12}$ and $X_{16}$.
Resolutions of type \eqref{eq:pomembno} can be writen for bundles with  
$c_1=0$, $c_2=4$ and $c_1=1$, $c_2=8$ on $X_{12}$ and for $c_1=2$, $c_2=16$ on $X_{16}$.

In the next section we will also need the following interesting example, when $X$ is CICY of type (2,2,3) and $\cE$ is a bundle with $c_1=0$ and $c_2=4$.
In this case we have
$$0\longrightarrow \cO_Y(-3) \to \cO_Y(-2)^3 \longrightarrow \cO_Y(-1)^3 \longrightarrow \cI_C \longrightarrow 0$$
and
\begin{equation}\label{X12}
0\longrightarrow \cO_Y(-3)\oplus \cO_Y(-2)^3 \longrightarrow \cO_Y\oplus \cO_Y(-1)^3 \longrightarrow \cE \longrightarrow 0,
\end{equation}
where $Y$ is a complete intersection of type (2,2) in $\PP^6$.

\begin{remark}
A smooth elliptic curve $C$ of degree 5 exists on $X_{12}$ by \cite{knu}. $C$ corresponds to an indecomposable rank 2 ACM bundle $\cE$ with $c_1=0$, $c_2=5$ ($C$ is a locally complete intersection since it is smooth and it is subcanonical since it has trivial canonical sheaf). $C$ is non-degenerate in $\PP^4$ and has five quadrics for minimal generators (see \eqref{g1d5}). Its minimal resolution in a fourfold $Y$ of type (2,2) in $\PP^6$ is 
$$0 \to \cO_Y(-4)\to \cO_Y(-2)^3\oplus \cO_Y(-3)^2\to \cO_Y(-2)^3\oplus \cO_Y(-1)^2\to \cI_C\to 0.$$
We can easily see that $\cE$ is not of type \eqref{eq:pomembno}, despite $C$ being AG in a fourfold of type (2,2) in $\PP^6$. Thus the assumption $c_1(\cE)=c-d_i$ in Theorem \ref{izjave} is neccesary (in this case $c_1(\cE)=0$, $c=4$, $d_i=3$). 

A Similar result holds for a bundle corresponding to a smooth elliptic curve $C$ of degree 6 on $X_{16}$ (by \cite{fis} $C$ has nine quadrics for minimal generators). 
\end{remark}

\section{Proof of Theorem \ref{theorem1}}

In the previous section we proved the part of the Theorem \ref{theorem1} regarding classification. We found the bundles with $c_1=-2$ and $c_1=-1$. Bundles with $c_1=0$ correspond to elliptic curves.
Knutsen \cite{knu} showed the  existence of smooth elliptic curves on all $X_r$ of degree $d\geq 3$, except for $d=3$ on $X_{16}$. By Eisenbud \cite{eis} smooth elliptic curves are ACM. Since they are smooth and subcanonical, the Serre correspondence gives the existence of ACM bundles with $c_1=0$ for all $c_2$ listed in Theorem \ref{theorem1}.

Bundles with $c_1=1$ correspond to canonical curves.
Knutsen showed the existence of smooth curves of degree 10 and genus 6 on $X_8$ and $X_9$ and smooth curves of degree 12 and genus 7 on $X_{12}$. These curves are canonical and thus by the Noether-Enriques-Petri theorem they are ACM. Furthermore, following Noether \cite{noe}, Schreyer \cite{sch} wrote the minimal resolutions of these curves in $\PP^{g-1}$
$$0 \to \cP_{g-2}\to \cdots \to \cP_1\to \cI_C\to 0,$$
where $g$ is the genus of $C$.  
It follows that the canonical curves are AG, thus they are subcanonical. The Serre corespondence then implies the existence of ACM bundles with 
$c_1=1$, $c_2=12$ on $X_{12}$ and $c_1=1$, $c_2=10$ on $X_8$ and $X_9$.

In the sequel we will prove the existence of the remaining bundles from Theorem \ref{theorem1}.

In order to obtain the last three cases of bundles with $c_1=1$ we restrict the exact sequence \eqref{eq:pomembno} to $X$ to obtain 
\begin{equation}\label{koneckonec}
0 \to \cE(-d)\to \cL_0^{\vee}(c-2d)\otimes \cO_X\to \cL_0\otimes \cO_X \to \cE \to 0.
\end{equation}

First we prove the existence of a bundle with $c_1=1$ and $c_2=6$ on $X_8$. Recall the minimal resolution \eqref{X8} of a bundle with $c_1=-1$ and $c_2=2$ on $X_8$. As in \eqref{koneckonec} we get an exact sequence 
$$0\longrightarrow \cE(-4) \longrightarrow \cO_{X_8}(-5)\oplus \cO_{X_8}(-3)^3 \longrightarrow \cO_{X_8} \oplus \cO_{X_8}(-2)^3 \longrightarrow \cE \longrightarrow 0,$$
and from this we get a short exact sequence
\begin{equation}\label{zadnje} 
0\longrightarrow \cE(-4) \longrightarrow \cO_{X_8}(-5)\oplus \cO_{X_8}(-3)^3 \longrightarrow \cF(-3) \longrightarrow 0,
\end{equation}
where $\cF$ is a rank 2 bundle on $X_8$ by the Auslander-Buchsbaum formula. From \eqref{zadnje} also follows that $\cF$ is indecomposable and normalized with $c_1(\cF)=1$ and $c_2(\cF)=6$. The bundle $\cF$ exists on $X_8$ since $\cE$ exists.

Next we prove the existence of a bundle with $c_1=1$ and $c_2=6$ on $X_9$. Consider the minimal resolution \eqref{X9} of a bundle with $c_1=0$ and $c_2=3$ on $X_9$. As above we get a short exact sequence
$$0\longrightarrow \cE(-3) \longrightarrow \mathcal{O}_{X_9}(-3)\oplus \cO_{X_9}(-2)^3 \longrightarrow \cF(-2) \longrightarrow 0,$$
where $\cF$ is a normalized indecomposable rank 2 bundle with $c_1=1$ and $c_2=6$ on $X_9$.

Finally the existence of a bundle with $c_1=1$ and $c_2=8$ on $X_{12}$ will be obtained from the minimal resolution \eqref{X12} of a bundle with $c_1=0$ and $c_2=4$ on $X_{12}$. Again, we get a short exact sequence
$$0\longrightarrow \cE(-3) \longrightarrow \mathcal{O}_{X_{12}}(-3)\oplus \cO_{X_{12}}(-2)^3 \longrightarrow \cF(-2) \longrightarrow 0,$$
with $\cF$ a normalized indecomposable rank 2 bundle with $c_1=1$ and $c_2=8$ on $X_{12}$. 
This finishes the proof of Theorem \ref{theorem1}.

We conclude the paper by proving the existence of an indecomposable bundle of higher rank.

Canonical curves which correspond to bundles with $c_1=1$, $c_2=10$ on $X_8$ have six minimal quadric generators (see e.g. \cite{bab}) and thus the minimal resolution of the ideal sheaf $\cI_C$ is
$$0\to \cO_Y(-5)\longrightarrow \cO_Y(-3)^5 \longrightarrow  \cO_Y(-2)^5 \longrightarrow \cI_C \longrightarrow 0,$$
where $Y$ is fourfold of degree 4 in $\PP^5$.
Theorem \ref{izjave} gives a minimal resolution of a bundle $\cE$ with $c_1=1$ and $c_2=10$: 
$$0\to \cO_Y(-3)\oplus \cO_Y(-2)^5 \longrightarrow  \cO_Y\oplus \cO_Y(-1)^5 \longrightarrow \cE \longrightarrow 0.$$
As above, we construct a short exact sequence 
$$0\to \cE(-4) \longrightarrow  \cO_{X_8}(-3)\oplus \cO_{X_8}(-2)^5 \longrightarrow \cF(-2) \longrightarrow 0.$$
It can be easily seen that $\cF$ is either an indecomposable rank 4 vector bundle on $X_8$ or a direct sum of a line bundle and an indecomposable rank 3 bundle.
\begin{center}
ACKNOLEDGMENT
\end{center}
I would like to thank to A. Buckley for detailed study of the material and giving many valuable comments. I am also grateful to D. Faenzi, T. Koďż˝ir, N. Mohan Kumar and G. Ravindra for many useful explanations.

\end{document}